\documentclass[12pt]{amsart}
\usepackage{amssymb}
\newcommand{\R}{\mathbb{R}}
\newcommand{\N}{\mathbb{N}}
\newcommand{\Z}{\mathbb{Z}}
\newcommand{\C}{\mathbb{C}}

\newcommand{\ds}{\displaystyle}
\newtheorem{defin}{Definition}

\newtheorem{lemme}{Lemma}
\newtheorem{coro}{Corollary}
\newtheorem{theorem}{Theorem}

\begin{document}

\title [Tur\'an Type Inequalities ]{ Tur\'an Type Inequalities for Dunkl Kernel and $q$-Dunkl Kernel \\}%

\author[  K. Mehrez, M. Ben Said, J. El Kamel]{   Khaled Mehrez, Mariem Ben Said  and Jamel El Kamel  }
 \address{Jamel El Kamel. D\'epartement de Math\'ematiques fsm. Monastir 5000, Tunisia.}
 \email{jamel.elkamel@fsm.rnu.tn}
 \address{Khaled Mehrez. D\'epartement de Math\'ematiques ISSAT.Kasserine , Tunisia.}
 \email{k.mehrez@yahoo.fr}
 \address{Mariem Ben Said . D\'epartement de Math\'ematiques ISMAIK.Kairouan 3100, Tunisia.}
 \email{bensaid.mery@gmail.com}
\begin{abstract}
We prove tur\'an type inequalities for Dunkl kernel.
 We provide a $q$-integral representation for the $q$-Dunkl kernel. Using a $q$-version of Schwartz inequality, 
 we get a tur\'an type inequalities for $q$-Dunkl kernel.
 \end{abstract}
\maketitle
%%%%%%%%%%%%%%%%%%%%%%%%%%%%%%%%%%%%%%%%%%%%%%%%%%%%%%%
{\it keywords:}  Dunkl kernel;  $q$-Dunkl kernel; Tur\'an type inequalities.\\
MSC (2010) 33C10, 33C52,  33D05, 39B62.\\
\section{Introduction}
\noindent In 1941, Paul Tur\'an  established the famous Tur\'an inequality for Legendre polynomials.
$$ \displaystyle P_{n-1}(x) P_{n+1}(x) <\left|P_n(x) \right|^2,\quad \left|x\right|<1, \quad n=1,2,...$$
In 1948, Gabor Szeg\"o  presented elegant proofs of Tur\'an inequality for Legendre polynomials 
and extented the result to Gegenbauer, Laguerre and Hermite polynomials.\\
After 1948 analogous results were obtained by  several authors for a large class of orthogonal 
polynomials and special functions (for example Bessel, $q$-Bessel, modified Bessel, polygamma, 
Riemann Zeta functions). In 1981 one of the PhD student of P. Tur\'an, L. Alp\'ar \cite{Al} in Tur\'ans 
bibliography mentioned that the above Tur\'an inequality had a wide ranging effect. Actually, 
the Tur\'an type inequalities have a more extensive literature and recently the results have 
been applied in problems  arising from many fields such as  information theory, economic theory and biophysics. Recently it has been shown by \'A. Baricz \cite{B1,B2,B3} that the Gauss and Kummer hypergeometric functions, 
as well as the generalized hypergeometric functions satisfy some Tur\'an type inequalities. 
For deep study about  this subject we refer to \cite{AF,B2,B3,GR,KC,KoT}.\\

\noindent In this paper our aim is to provide some new Tur\'an type inequalities for
Dunkl kernel and $q$-Dunkl kernel.\\

\noindent Our paper is organized as follows : in section 2, we present some preliminary  results 
and notations that will be useful in the sequel. In section 3, using the series expansion of the 
Dunkl kernel $E_\nu(\lambda,x)$, we prove that the function $\nu\longmapsto E_\nu(\lambda,x)$ is log-convex 
on $(0,\infty)$. In particular we deduce some Tur\'an type inequalities for the 
Dunkl kernel. Using an integral representation, we show analogous results for 
 the normalized Dunkl kernel $\widetilde{E_\nu}(\lambda,x)$.
In section 4, using the series expansion of the 
q-Dunkl kernel $E_\nu(x,q^2)$, we prove that the function $\nu\longmapsto E_\nu(x,q^2)$ is log-convex 
on $]0,\infty[$, in particular we deduce some Tur\'an type inequalities for the 
q-Dunkl kernel. We establish a q-integral representation for  $q$-Dunkl kernel . Using a $q$-version of 
Schwartz inequality, we deduce some Tur\'an type inequalities for the the normalized $q$-Dunkl kernel. As application, in section 5, we give some hyperbolic Jordan's type inequalities for hyperbolic functions.

\section{Notations and preliminaries}
\noindent The Euler  gamma function $\Gamma(z)$ is defined for $\mathcal{R}(z)>0$, by 
$$\displaystyle \Gamma(z)=\int_0^\infty t^{z-1}e^{-t} dt.$$
The psi(or digamma) function $\psi(z)$ is the logarithmic derivative of $\Gamma(z)$, that is, 
 $$\displaystyle \psi(z)=\frac{\Gamma'(z)}{\Gamma(z)}.$$
It's well known tha the  digamma function satisfies 
$$\displaystyle \psi(x)=-\gamma+(x-1) \sum_{k\geq 0} \frac{1}{(k+1)}\frac{1}{(x+k)},\quad x>0,$$
where $\gamma$ is the Euler constant. Thus the digamma function is concave in $]0,\infty[$.\\

\noindent Throughout the section 4, we will fix $q \in ]0,1[$.
 We recall some usual notions and notations used in the
$q$-theory (see \cite{GR} and \cite{KC}).\\
We refer to the book by  G. Gasper and M. Rahman \cite{GR}, for
the notations,  definitions and  properties of the $q$-shifted
factorials and $q$- hypergeomtric functions.\\
We note 
$$\displaystyle\R_{q,+} = \{q^n: n\in \Z\}.
$$
The $q$-derivative $D_qf$  of a function $f$ is given by
$$\displaystyle
(D_qf)(x)={{f(x)-f(qx)}\over{(1-q)x}},~~  {\rm if}~~ x\not=0,
$$
$(D_qf)(0)=f'(0)$  provided $f'(0)$ exists. If  $f$ is
differentiable, then $(D_qf)(x)$  tends to $f'(x)$ as $q$ tends to
1.\\
 The $q$-Jackson integrals from $0$ to $a$,  from $0$ to $\infty$ and in a generic interval $[a,b]$ are
defined by (see \cite {Jac})
$$\displaystyle
\int_0^{a}{f(x)d_qx} =(1-q)a\sum_{n=0}^{\infty}{f(aq^n)q^n},
$$
$$\displaystyle
\int_0^{\infty}{f(x)d_qx}
=(1-q)\sum_{n=-\infty}^{\infty}{f(q^n)q^n},
$$
provided the sums converge absolutely, and
$$
\int_a^{b}{f(x)d_qx} =\int_0^{b}{f(x)d_qx}-\int_0^{a}{f(x)d_qx}.
$$
 The improper integral is defined in the following way
(see \cite {KC})
$$\displaystyle
\int_0^{{\infty}\over A}{f(x)d_qx}
  =(1-q)\sum_{n=-\infty}^{\infty}{f\left({{q^n}\over A}\right){{q^n}\over A}
  }.
$$
 We remark that for $n\in \Z$, we have
$$\displaystyle
\int_0^{{\infty}\over {q^n}}{f(x)d_qx}=\int_0^{\infty}{f(x)d_qx}.
$$

\indent $q$-Analogues of the exponential functions ( see \cite
{GR, KC}) is given by:
$$\displaystyle
e(q,z)=_1{\varphi}_0(0;-;q,z)=\sum_{n=0}^{\infty}{(1-q)^n\over{(q;q)_n}}z^n
= {1\over{(z;q)_{\infty}}}.
$$
For the convergence of the  series, we need $\ds \mid
 z\mid<1$; however, due to its product
 representation, $ \ds e_q$ is continuable to a meromorphic function on $\C$ and has  simple
 poles at $\ds z= q^{-n},~~n\in\N$.
 $$\displaystyle E(q;z)=\sum_{n=0}^\infty q^{\frac{n(n-1)}{2}} \frac{z^n}{(q;q)_n}=\prod_{k=0}^\infty (1+z^k), \quad z\in \mathbb{C}.$$
 
\noindent Jackson \cite {Jac} defined a $q$-analogous of the Gamma
function by
$$\displaystyle
\Gamma_q (x) ={(q;q)_{\infty}\over{(q^x;q)_{\infty}}}(1-q)^{1-x} ,
\qquad x\not={0},{-1},{-2},\ldots.
$$
It is well known that it satisfies:
$$\displaystyle
\Gamma_q(x+1)=\frac{1-q^x}{1-q} \Gamma_q(x),\quad
\Gamma_q(1)=1~~~~{\rm and}~~~~
\lim_{q\rightarrow1^-}\Gamma_q(x)=\Gamma(x),~~ \Re(x)>0.
$$
The $q$-psi (or $q$-digamma) function is defined as the Logarithmic $q$-derivative of the $q$-gamma function:      
$$\psi_{q}(x)=\frac{\Gamma_{q}'(x)}{\Gamma_{q}(x)},$$ 
and satisfies:
$$\psi_{q}(x)=-Log(1-q)+Log(q)\sum_{n=1}^{\infty}\frac{q^{nx}}{1-q^{n}}; \quad q\in ]0,1[,$$
where $Log(x)$ means $Log_{e}(x).$

The $q$-modified Bessel function of first kind is defined by ( see \cite{BGR} ) :
$$\displaystyle
I_\nu^{(1)}\left((1-q^2)z;q^2\right)=\frac{1}{\Gamma_{q^2}(\nu+1)}\sum_{k=0}^\infty 
\frac{(1-q^2)^{2k}z^{2k+\nu}}{2^{2k+\nu}(q^2;q^2)_k(q^{2\nu+1};q^2)_k},\quad \left|z\right|<\frac{1}{1-q^2}.
$$
The normalized $q$-modified Bessel function of the first kind is defined by :

$$\displaystyle
\mathcal{I}_\nu(z;q^2)=(1+q)^\nu \frac{\Gamma_{q^2}(\nu+1)}{z^\nu} I_\nu^{(1)}\left(2(1-q)z;q^2\right).
$$
\section{Tur\'an Type Inequalities for Dunkl kernel}
We recall that the Dunkl operator is defined for $f\in \mathcal{C}^1(\mathbb{R})$ by :
\begin{equation}
 T_\nu f(x)=f'(x)+\frac{\nu}{x}\left[\frac{f(x)-f(-x)}{2}\right], \quad \nu>0.
\end{equation}
For $\lambda \in\mathbb{C}$, the Dunkl kernel $E_\nu(\lambda,.)$ on $\mathbb{R}$ was introduced by C. Dunkl in \cite{D2}  and is given by :
 \begin{equation}
 E_\nu(\lambda,x)=j_{\nu-\frac{1}{2}}(i\lambda x)+\frac{\lambda x}{2\nu+1}j_{\nu+\frac{1}{2}}(i\lambda x),
\end{equation}
where $j_\alpha$ is the normalized Bessel function of the first kind of order $\alpha$ . The Dunkl
kernel $E_\nu(\lambda,x)$ is the unique solution on $\mathbb{R}$ of the initial problem associated to Dunkl operator : 
\begin{equation}
\displaystyle T_\nu f(x)=\lambda f(x), \quad f(0)=1,\quad x\in \mathbb{R}.
\end{equation}

\begin{lemme}{\rm(see \cite{S})}. For $\lambda,x\in \mathbb{R}$ and $\nu>0$,
 the  Dunkl kernel $E_\nu(\lambda,x)$ admits the  series expansions 
\begin{equation}
 E_\nu(\lambda,x)=\sum_{n=0}^\infty \frac{(\lambda x)^n}{b_n(\nu)},
 \end{equation}
where 
\begin{equation}
 \displaystyle b_{2n}(\nu)=2^{2n}n!\frac{\Gamma (n+\nu+1)}{\Gamma (\nu+1)},\qquad b_{2n+1}(\nu)=2(\nu+1) b_{2n}(\nu+1).
 \end{equation}
\end{lemme}
\begin{lemme} For $\lambda,x\in \mathbb{R}$ and $\nu>0$, the  Dunkl kernel $E_\nu(\lambda,x)$ admits the 
 following integral representation 
 \begin{equation}
  \displaystyle E_\nu(\lambda,x)=c(\nu)\int_{-1}^1 e^{\lambda xt} (1-t^2)^{\nu-1}(1+t) dt,
  \end{equation}
 where $\displaystyle c(\nu)=\frac{\Gamma(\nu+\frac{1}{2})}{\Gamma(\frac{1}{2})\Gamma(\nu)}$.
\end{lemme}

\begin{theorem}
For $\lambda, x\in \mathbb{R}^+$, the function $\nu\longmapsto  E_\nu(\lambda,x)$ is log-convex on $]0,\infty[$\\
i.e 
\begin{equation}E_{\alpha \nu_1+(1-\alpha)\nu_2}(\lambda,x)\leq \left[E_{\nu_1}(\lambda,x)\right]^\alpha
\left[E_{\nu_2}(\lambda,x)\right]^{1-\alpha},\quad \forall \nu_1>0, \nu_2>0, \quad \forall\alpha\in [0,1].
\end{equation}
In particular, we get  Tur\'an type inequalities for Dunkl kernel :

\begin{equation} \left[E_{\nu+1}(\lambda,x)\right]^2\leq E_\nu(\lambda,x)E_{\nu+2}(\lambda,x),\quad \forall \nu>0.
\end{equation}

\end{theorem}
\begin{proof}
To show the Log-convexity of the function $\nu\longmapsto  E_\nu(\lambda,x)$, we just need to show the Log-convexity of each term of its  series expansion and then, we use the fact that sums of Log-convex functions are Log-convex too.\\
 Let $n\geq 0$, since the function $\psi$ is concave on $]0,\infty[$, we get :
$$\frac{d^{2}}{d\nu^{2}}[Log(\frac{1}{b_{2n}(\nu)})]= \psi^{'}(\nu+1)-\psi^{'}(\nu+n+1)\geq 0 $$
and $$\frac{d^{2}}{d\nu^{2}}[Log(\frac{1}{b_{2n+1}(\nu)})]= \psi^{'}(\nu+2)-\psi^{'}(\nu+n+2)+\frac{1}{(\nu+1)^{2}}\geq 0 .$$
\\
\noindent Thus for $\nu_{1}, \nu_{2}>0,\quad \alpha \in [0,1]$ :
$$ {E}_{\alpha \nu_1+(1-\alpha)\nu_2}(\lambda,x)\leq \left[{E}_{\nu_1}(\lambda,x)\right]^\alpha
\left[{E}_{\nu_2}(\lambda,x)\right]^{1-\alpha}.$$
\noindent In particular; for $\nu>0,\quad \nu_{1}=\nu ,\quad \nu_{2}=\nu+2 $ and $ \alpha =\frac{1}{2}$ , Tur\'an type inequality for Dunkl Kernel is deduced.
\end{proof}
\begin{defin} For $\lambda,x\in \mathbb{R}$ and $\nu>0$,  the normalized Dunkl kernel is defined by 
\begin{equation}
\widetilde{E_\nu}(\lambda,x)=\frac{1}{c_\nu}E_\nu(\lambda,x),
\end{equation}
where $c_\nu$ is given by (7).
\end {defin}
\begin{theorem}
For $\lambda, x \in \mathbb{R}$ , the function $\nu\longmapsto  \widetilde{E_\nu}(\lambda,x)$ is log-convex on $]0,\infty[$\\
i.e 
\begin{equation}
\widetilde{E}_{\alpha \nu_1+(1-\alpha)\nu_2}(\lambda,x)\leq \left[\widetilde{E}_{\nu_1}(\lambda,x)\right]^\alpha
\left[\widetilde{E}_{\nu_2}(\lambda,x)\right]^{1-\alpha},\quad \forall \nu_1>0, \nu_2>0, \quad \forall\alpha\in [0,1].
\end{equation}
In particular; we get  Tur\'an type inequalities for normalized Dunkl kernel :

\begin{equation} \left[\widetilde{E}_{\nu+1}(\lambda,x)\right]^2\leq\widetilde{E}_\nu(\lambda,x)\widetilde{E}_{\nu+2}(\lambda,x),\quad \forall \nu>0.
\end{equation}

\end{theorem}
\begin{proof}
Using the integral  representation  (6) of the Dunkl Kernel and H\"older inequality we have,
for $\nu_{1} , \nu_{2}>0$ and $\alpha \in ]0,1[$
$$
\widetilde{E}_{\alpha \nu_1+(1-\alpha)\nu_2}(\lambda,x)=\int_{-1}^1 e^{\lambda xt} 
(1-t^2)^{\alpha \nu_1+(1-\alpha)\nu_2} (1+t) dt \qquad\qquad\qquad\qquad\qquad\qquad$$
$$\qquad\qquad\qquad\quad= \int_{-1}^1
\left[ e^{\lambda xt} (1-t^2)^{\nu_1-1}(1+t)\right]^\alpha \left[ e^{\lambda xt} (1-t^2)^{\nu_2-1}(1+t)\right]^{1-\alpha} dt $$
$$\qquad\qquad\qquad\qquad\quad\leq \left[\int_{-1}^1 e^{\lambda xt} (1-t^2)^{\nu_1-1}(1+t)dt\right]^\alpha  \left[\int_{-1}^1 e^{\lambda xt} (1-t^2)^{\nu_2-1}(1+t)dt\right]^{1-\alpha}
$$
$$\leq \left[\widetilde{E}_{ \nu_1}(\lambda,x)\right]^\alpha  \left[\widetilde{E}_{ \nu_2}(\lambda,x)\right]^{1-\alpha}.\qquad\qquad\qquad$$
In particular; for $\nu>0 , \quad \nu_1=\nu , \quad \nu_2=\nu+2 $ and $\alpha =\frac{1}{2}$, we get Tur\'an type inequalities for the normalized Dunkl Kernel.
\end{proof}
\section{Tur\'an Type Inequalities for $q$-Dunkl kernel}
We consider the $q$-Dunkl operator $T_{q,\nu}$ defined by :
\begin{equation}
T_{q,\nu}f(x)=D_qf(x)+\frac{[2\nu+1]_q}{x}\left[\frac{f(qx)-f(-qx)}{2}\right].
\end{equation}
We note that the $q$-Dunkl operator $T_{q,\nu}$ tends to the Dunkl operator $T_\nu$ as $q\rightarrow 1^-$.

\begin{defin} {\rm($q$-Dunkl kernel)}\\

For $q\in]0,1[$ and $\displaystyle |x|<\frac{1}{(1-q)^2}$, we define the $q$-Dunkl kernel by 
\begin{equation}
E_\nu(x;q^2)=\mathcal{I}_\nu(x;q^2)+\frac{x}{(1+q)\left[\nu+1\right]_{q^2}}\mathcal{I}_{\nu+1}(x;q^2),
\end{equation}
where $\mathcal{I}_\nu(x;q^2)$ is the normalized $q$-modified Bessel function of the first kind. 
\end{defin}
\begin{lemme}
For $q\in ]0,1[, \, |\lambda|< \frac{1}{1-q}$, the q-Dunkl kernel $E_\nu(\lambda .;q^2)$ is the unique 
analytic solution of the q-problem

\begin{equation}
T_{q,\nu}f(x)=\lambda f(x),\quad f(0)=1.
\end{equation}

\end{lemme}

\begin{lemme}
For $\displaystyle |x|<\frac{1}{(1-q)^2}$, The $q$-Dunkl kernel $E_\nu(x;q^2)$ admits the series expansions 
\begin{equation}
E_\nu(x;q^2)=\sum_{k=0}^\infty \frac{x^{k}}{b_k(\nu;q^2)},
\end{equation}
where 
\begin{equation}
b_{2k}(\nu;q^2)=\frac{(1+q)^{k}\Gamma_{q^2}(k+1)\Gamma_{q^2}(\nu+k+1)}{\Gamma_{q^2}(\nu+1)}
\end{equation}

\begin{equation}
b_{2k+1}(\nu;q^2)=\frac{(1+q)^{2k+1}\Gamma_{q^2}(k+1)\Gamma_{q^2}(\nu+k+2)}{\Gamma_{q^2}(\nu+1)}
\end{equation}

\end{lemme}

\begin{theorem}
For $q\in ]0,1[$ and $x, \, \lambda \in [0,\frac{1}{1-q}[$, the function 
$\nu \longmapsto E_\nu(\lambda x;q^2)$ is log-convex on $]0,\infty[$,\\
i.e
\begin{equation}
E_{\alpha\nu_1+(1-\alpha)\nu_2}(\lambda x;q^2)\leq 
\left[E_{\nu_1}(\lambda x;q^2)\right]^{\alpha}\left[E_{\nu_2}(\lambda x;q^2)\right]^{1-\alpha},\quad 
\forall \nu_1,\nu_2>0,\quad \forall \alpha \in[0,1].
\end{equation}
In particular; the Tur\'an type inequalities for The $q$-Dunkl kernel holds :
\begin{equation}
\left[E_{\nu+1}(\lambda x;q^2)\right]^2\leq E_{\nu}(\lambda x;q^2)E_{\nu+2}(\lambda x;q^2), \quad \forall\nu>0.
\end{equation}
\end{theorem}

\begin{proof}
As in the classical case, we establish the Log-convexity of the function \\
$\nu \longmapsto E_{\nu}(\lambda x;q^2)$ by proving the Log-convexity of each term of its series expansion, and then we use the fact that the sums of Log-convex functions are Log-convex too.
Let $k\geq 0$, Since $\psi^{'}_q$ is decreasing on $]0,\infty[$, we get:
$$\frac{d^2}{d\nu^2}\left[Log\left(\frac{1}{b_{2k}(\nu;q^2)}\right)\right]=\psi^{'}_{q^2}(\nu+1)- \psi^{'}_{q^2}(\nu+k+1)\geq 0$$
and 
$$\frac{d^2}{d\nu^2}\left[Log\left(\frac{1}{b_{2k+1}(\nu;q^2)}\right)\right]=\psi^{'}_{q^2}(\nu+1)- \psi^{'}_{q^2}(\nu+k+2)\geq 0$$
Consequently, the function $\nu \longmapsto E_{\nu}(\lambda x;q^2)$ is Log-convex on $]0,\infty[$ : 
$$E_{t\nu_1+(1-\alpha)\nu_2}(\lambda x;q^2)\leq 
\left[E_{\nu_1}(\lambda x;q^2)\right]^\alpha\left[E_{\nu_2}(\lambda x;q^2)\right]^{1-\alpha},\quad
\forall \nu_1,\nu_2>0,\quad \forall \alpha\in[0,1].$$
In particular; for $\alpha=\frac{1}{2} \quad \nu_1=\nu, \quad \nu_2=\nu+2,$ Tur\'an types inequality for $q$-Dunkl Kernel holds.
\end{proof}
\begin{lemme} 
For all $q\in ]0,1[$, $\displaystyle |x|<\frac{1}{(1-q)^2}$ , the $q$-Dunkl kernel admits the q-integral representation : 
\begin{equation}
E_{\nu}(x;q^2)=C(\nu,q^2)\int_{-1}^1 W_\nu(t;q^2)(1+t) E(q, (1-q)tx) d_qt,
\end{equation}
where 
\begin{equation}
C(\nu,q^2)=\frac{(1+q)\Gamma_{q^2}(\nu+1)}{2\Gamma_{q^2}(\frac{1}{2})\Gamma_{q^2}(\nu+\frac{1}{2})}
\end{equation}
and 
\begin{equation}
W_\nu(x;q^2)=\frac{(x^2q^2;q^2)_\infty}{(x^2q^{2\nu+1};q^2)_\infty}.
\end{equation}
\end{lemme}
\begin{proof}
Let $q\in ]0,1[,\quad \displaystyle |x|<\frac{1}{(1-q)^2}$, as in \cite{OR1}, the  normalized $q$-modified Bessel function of first the kind admits the following integral representation :
\begin{equation}
I_{\nu}(x;q^2)= C(\nu,q^2)\int_{-1}^1 W_\nu(t;q^2) E(q, (1-q)tx) d_qt,  
\end{equation}
where $C(\nu,q^2) ; W_\nu(t;q^2)$ are given respectively by (21) and (22).\\
Knowing that, 
\begin{equation}
D_q(E((1-q)x;q))=E((1-q)x;q)
\end{equation}
and since
\begin{equation}
E_{\nu}(x;q^2)=I_{\nu}(x;q^2)+D_q(I_{\nu}(x;q^2)).
\end{equation}
Using (23), (24) and (25), we get :
$$
E_{\nu}(x;q^2)=I_{\nu}(x;q^2)+D_q(I_{\nu}(x;q^2))$$
$$=C(\nu,q^2)\int_{-1}^1 W_\nu(t;q^2) E(q, (1-q)tx) d_qt
+C(\nu,q^2)\int_{-1}^1 W_\nu(t;q^2) D_q[E(q, (1-q)tx)] d_qt$$
$$=C(\nu,q^2)\int_{-1}^1 W_\nu(t;q^2) E(q, (1-q)tx) d_qt+
C(\nu,q^2)\int_{-1}^1 W_\nu(t;q^2)t E(q, (1-q)tx)d_qt.
$$

\end{proof}

\begin{defin}
For all $q\in ]0,1[$ and $\displaystyle |x|< \frac{1}{(1-q)^2}$, the normalized $q$-Dunkl kernel 
is defined by 
 \begin{equation}
 \widetilde{E_\nu}(x;q^2)=\frac{E_\nu(x;q^2)}{C(\nu,q^2)},
 \end{equation}
 where $C(\nu,q^2)$ is given by (21).
\end{defin}
\begin{theorem}
For $q\in ]0,1[$, $\displaystyle |x|< \frac{1}{(1-q)^2}$ , the the normalized $q$-Dunkl kernel 
$ \tilde{E}_\nu(x;q^2)$ satisfy a Tur\'an type inequality ,\\
i.e

\begin{equation}
\displaystyle \left[\widetilde{E}_{\nu+1}(x;q^2)\right]^2\leq \widetilde{E}_{\nu}(x;q^2)\widetilde{E}_{\nu+2}(x;q^2), \quad \forall\nu>0.
\end{equation}
\end{theorem}
\begin{proof} 

Let $q\in ]0,1[$, $\nu >0$ and $\displaystyle |x|< \frac{1}{(1-q)^2}$. \\
Using the relation :
$$\displaystyle (q^{2\nu+1}x^2;q^2)_\infty=(1-x^2q^{2\nu+1})(q^{2\nu+3}x^2;q^2)_\infty,$$ 
we get : 
\begin{equation}
W_{\nu+1}(x,q^2)<\left[W_\nu(x,q^2)\right]^{\frac{1}{2}}\left[W_{\nu+2}(x,q^2)\right]^{\frac{1}{2}},
\end{equation}
where $W_\nu(x,q^2)$ is given by (22).\\
Using (28), the $q$-version of Schwatz inequality and the $q$-integral representation of the $q$-Dunkl Kernel, we obtain:\\
for $\nu >0,\, q\in ]0,1[$ and $\displaystyle |x|<\frac{1}{(1-q)^2}$
$$
\widetilde{E}_{\nu+1}(x;q^2)=\int_{-1}^1 W_{\nu+1}(t;q^2) (1+t)E_q((1-q)tx) d_qt$$
$$\leq \int_{-1}^1 [ W_{\nu}(t;q^2) (1+t)E_q((1-q)tx)]^{\frac{1}{2}}\times [ W_{\nu+2}(t;q^2) (1+t)E_q((1-q)tx)]^{\frac{1}{2}}d_qt$$
$$\leq \left[\int_{-1}^1 W_{\nu}(t;q^2) (1+t)E_q((1-q)tx)\right]^{\frac{1}{2}}
\times \left[\int_{-1}^1 W_{\nu+2}(t;q^2) (1+t)E_q((1-q)tx)\right]^{\frac{1}{2}}$$
$$\leq \left[\widetilde{E}_{\nu}(x;q^2)\right]^{\frac{1}{2}} \left[\widetilde{E}_{\nu+2}(x;q^2)\right]^{\frac{1}{2}}.$$
\end{proof}
\section{Applications}

\begin{theorem}
The following assertions are true :\\

(1) For  $\lambda,\, x \geq 0,$ the function $\displaystyle \nu\longmapsto \frac{E_{\nu+1}(\lambda,x)}{E_{\nu}(\lambda,x)}$ is increasing on $]0,\infty[$.\\

(2) For $q\in ]0,1[$ and $  x, \,\lambda \in [0, \frac{1}{1-q}[$, the function $\displaystyle \nu\longmapsto \frac{E_{\nu+1}(x,q^{2})}{E_{\nu}(x,q^{2})}$  is increasing on $]0,\infty[.$
\end{theorem}
\begin{proof}
From Theorems 1 and 3, we deduce the Log-convexity of the functions \\
$\displaystyle  \nu\longmapsto E_{\nu}(\lambda,x)$ and $\nu \longmapsto E_{\nu}(x,q^{2})$ on $]0,\infty[$. 
Thus, the functions 
 $\displaystyle  \nu\longmapsto Log \left[\frac{E_{\nu+1}(\lambda,x)}{E_{\nu}(\lambda,x)}\right]$ and 
 $\displaystyle  \nu \longmapsto Log \left[\frac{E_{\nu+1}(x,q^{2})}{E_{\nu}(x,q^{2})}\right] $
  are increasing on $]0,\infty[$. Which completes the proof.
\end{proof}

\noindent In the next corollary, we give some hyperbolic Jordan's type inequalities for hyperbolic  functions.
\begin{coro}
the following inequalities are valid :
 
$$\displaystyle  (1-x)e^{x}\leq \frac{\sinh x}{x} \quad ; x>0 $$
$$\displaystyle   (1+x)e^{-x}\leq \frac{\sinh x}{x} \quad ; x<0.$$
\end{coro}

\begin{proof}
Since the function $\nu\longmapsto \frac{E_{\nu+1}(\lambda,x)}{E_{\nu}(\lambda,x)}$ is increasing on $]0,\infty[$, we get :
$$\displaystyle \frac{E_{\nu+1}(\lambda,x)}{E_{\nu}(\lambda,x)} \leq 1 .$$ 
By the definition of the Dunkl Kernel and since :
$$\displaystyle  j_{\frac{-1}{2}}(ix)=\cosh x,$$
$$\displaystyle  j_{\frac{-1}{2}}(ix)=\frac{\sinh x}{x}, $$
$$\displaystyle  j_{\frac{3}{2}}(ix)= -3(\frac{\sinh x}{x^{3}}-\frac{\cosh x}{x^{2}}) $$
we conclude. 
The second inequality is deduced by parity. 
\end{proof}

\end{document}